\documentclass{amsart}[12pt,article]

\usepackage{amsmath,amssymb,amscd,amsthm,indentfirst}
\usepackage[all]{xy}
\usepackage{enumerate}
\usepackage{graphicx}

\usepackage[utf8]{inputenc}
\usepackage[T1]{fontenc}

\newtheorem{thm}{Theorem}[section]

\theoremstyle{definition}
\newtheorem{defn}[thm]{Definition}
\newtheorem{remark}[thm]{Remark}

\def\Hom{\mathrm{Hom}}

\def\Aut{\mathrm{Aut}}

\newcommand{\Syl}{\operatorname{Syl}\nolimits}

\makeatletter
\@tempcnta=\@ne
\def\@nameedef#1{\expandafter\edef\csname #1\endcsname}
\loop\ifnum\@tempcnta<27
  \@nameedef{C\@Alph\@tempcnta}{\noexpand\mathcal{\@Alph\@tempcnta}}
  \advance\@tempcnta\@ne
\repeat

\makeatletter
\@tempcnta=\@ne
\def\@nameedef#1{\expandafter\edef\csname #1\endcsname}
\loop\ifnum\@tempcnta<27
  \@nameedef{B\@Alph\@tempcnta}{\noexpand\mathbb{\@Alph\@tempcnta}}
  \advance\@tempcnta\@ne
\repeat

\makeatletter
\@tempcnta=\@ne
\def\@nameedef#1{\expandafter\edef\csname #1\endcsname}
\loop\ifnum\@tempcnta<27
  \@nameedef{F\@Alph\@tempcnta}{\noexpand\mathsf{\@Alph\@tempcnta}}
  \advance\@tempcnta\@ne
\repeat

\date{\today}
\title{A cohomological characterization of nilpotent fusion systems}

\author{Antonio D\'{i}az Ramos}
\author{Arturo Espinosa Baro }
\author{Antonio Viruel}
\address{Departamento de {\'A}lgebra, Geometr{\'\i}a y Topolog{\'\i}a,
Universidad de M{\'a}\-la\-ga, Apdo correos 59, 29080 M{\'a}laga,
Spain.}
\email{adiazramos@uma.es}
\address{Faculty of Mathematics and Computer Science, Adam Mickiewicz University, Umultowska 87, 61-614 Poznan, Poland.}
\email{arturo.espinosabaro@gmail.com}
\address{Departamento de {\'A}lgebra, Geometr{\'\i}a y Topolog{\'\i}a,
Universidad de M{\'a}\-la\-ga, Apdo correos 59, 29080 M{\'a}laga,
Spain.}
\email{viruel@uma.es}
\thanks{Authors partially supported by MEC grants MTM2013-41768-P and MTM2016-78647-P and Junta de Andaluc{\'\i}a grant FQM-213. Second author supported by Polish National Science Centre grant 2016/21/P/ST1/03460 within the European Union's Horizon 2020 research and innovation programme under the Marie Skłodowska-Curie grant agreement No. 665778. }

\begin{document}

\begin{abstract}
We provide a nilpotency criterion for fusion systems in terms of the vanishing of its cohomology with twisted coefficients.
\end{abstract}

\maketitle

\section{Introduction}
\label{section:introduction}

Let $G$ be a finite group and let $p$ be a prime. Then $G$ is said to be $p$-nilpotent if a Sylow $p$-subgroup has a complement, i.e., if there exists  a \emph{split} short exact sequence
\[
1\to N\to G\to S
\]
with $S\in \Syl_p(G)$. It turns out that this property can be characterized solely in terms of the \emph{fusion system} of $G$ over $S$. This terminology was introduced in \cite{BLO2} and it is straightforward that 
\[
\text{$G$ is $p$-nilpotent $\Leftrightarrow$ $\CF_S(G)=\CF_S(S)$,}
\]
where $\CF_S(G)$ denotes the the fusion system of $G$ over $S$. Consequently, a fusion system $\CF$ over the $p$-group $S$ is termed nilpotent if $\CF=\CF_S(S)$. Already several authors have provided fusion system counterparts to characterizations of $p$-nilpotency for finite groups, see \cite{BESW}, \cite{BGH}, \cite{CSV}, \cite{D}, \cite{DGPS}, \cite{GRV}, \cite{KLN} and \cite{LZ}. In this work, we prove the fusion system version of a $p$-nilpotency criterion from the late $60$'s due to Wong  \cite{W} and Hoechsmann, Roquette and Zassenhaus \cite{HRZ}.

\begin{thm}\label{thm:main}
Let $\CF$ be a fusion system. Then the following are equivalent:
\begin{enumerate}
\item[(1)] $\CF$ is nilpotent.
\item[(2)] For each $\BF_p[\CF]$-module $M$, if $H^m(\CF^c;M)=0$ for some $m>0$ then  $H^n(\CF^c;M)=0$ for every $n>0$.
\end{enumerate}
\end{thm}

Here, a $\BF_p[\CF]$-module is a finitely generated $\BF_p[S]$-module that is $\CF$-invariant, and $H^*(\CF^c;M)$ is twisted cohomology over $\CF^c$, i.e., over the $\CF$-centric subgroups. See Definitions \ref{def:Fmodule} and \ref{def:cohomologyFM} for full details. We give a topological proof of Theorem \ref{thm:main} via the classifying space $B\CF$ of $\CF$ \cite{C}.  Recall that, in the terminology of \cite{BLO2}, $B\CF\simeq |\CL|^\wedge_p$, where $\CL$ is the unique centric linking system associated to $\CF$ and $(\cdot)^\wedge_p$  denotes $p$-completion in the sense of Bousfield and Kan \cite{BK}. 

The original version of Theorem \ref{thm:main} was equivalently stated in terms of Tate's cohomology. Moreover, its proof resorted to dimension shifting. This approach is not suitable here as there are not enough acyclic modules. For instance, if $B\CF$ is simply connected, then, for each $\BF_p[\CF]$-module $M$, $H^*(\CF^c;M)=H^*(\CF;M)$ is cohomology with trivial coefficients. Hence, it will not vanish but in trivial cases.

\textbf{Notation:} Throughout this work, by fusion system we mean a \emph{saturated} fusion system. The unacquainted reader may find an explanation of this terminology and general background on fusion systems in \cite{AKO}.

\section{Proof of the theorem}
\label{section:proofofthetheorem}
We start introducing modules for fusion systems and their cohomology:
\begin{defn}\label{def:Fmodule}
Let $\CF$ be a fusion system over the finite $p$-group $S$. An $\BF_p[\CF]$-module is a finitely generated $\BF_p[S]$-module $M$ which is $\CF$-invariant, i.e., such that:
\[
\forall P\leq S\text{, }\forall \varphi\in \Hom_\CF(P,S)\text{, }\forall p\in P\text{, }\forall m\in M: \varphi(p)\cdot m=p\cdot m.
\]
\end{defn}

It is clear that an $\BF_p[\CF]$-module is the same thing as a finitely generated $\BF_p[S/\mathfrak{foc}(\CF)]$-module, where $\mathfrak{foc}(\CF)$ is the focal subgroup of $\CF$:
\[
\mathfrak{foc}(\CF)=\langle [P,\Aut_\CF(P)]\text{, } P\leq S\rangle.
\]
By inflation, every $\BF_p[\CF]$-module is also a $\BF_p[S/\mathfrak{hyp}(\CF)]$-module, where $\mathfrak{hyp}(\CF)$ is the hyperfocal subgroup of $\CF$:
\[
\mathfrak{hyp}(\CF)=\langle [P,O^p(\Aut_\CF(P))]\text{, } P\leq S\rangle.
\]

\begin{defn}[{\cite[Definition 2.3]{M}}]\label{def:cohomologyFM}
Let $\CF$ be a saturated fusion system over the finite $p$-group $S$ and let $M$ be an $\BF_p[\CF]$-module. For each $n\geq 0$, define the twisted cohmology group $H^n(\CF^c;M)$ as the $\CF^c$-stable elements:
\[
H^n(\CF^c;M)=\{z\in H^n(S;M)|\forall P\in \CF^c\text{, }\forall \varphi\in \Hom_\CF(P,S):res(z)=\varphi^*(z)\}.
\]
\end{defn}
Here, $\varphi^*:H^n(S;M)\to H^n(P;M)$ is the homomorphism induced in cohomology by $\varphi$, and $res=\iota^*$ for the inclusion $\iota\colon P\hookrightarrow S$. By Alperin's Fusion Theorem, if the action of $S/\mathfrak{hyp}(\CF)$ on $M$ is trivial, then $H^n(\CF^c;M)$ coincide with the $\CF$-stable elements 
\[
H^n(\CF;M)=\{z\in H^n(S;M)|\forall P\leq S\text{, }\forall \varphi\in \Hom_\CF(P,S):res(z)=\varphi^*(z)\}.
\]
In general, the abelian group $H^n(\CF^c;M)$
may be recovered via topology as the cohomology of the classifying space of $\CF$ \cite[Corollary 5.4]{M}:
\[
H^*(\CF^c;M)\cong H^*(B\CF;M),
\]
where $\pi_1(B\CF)=S/\mathfrak{hyp}(\CF)$ by \cite[Theorem B]{BCGLO}. Now we are ready to prove the main theorem.
\begin{proof}[Proof of Theorem \ref{thm:main}]
For the implication $(1)\Rightarrow (2)$, note that:
\[
H^n(\CF_S(S)^c;M)=H^n(S;M).
\]
From here it is enough to follow the group theoretical proof, see \cite[Theorem 1]{W} or \cite[Proposition 1a]{HRZ}. For the reverse implication, set $\pi=S/\mathfrak{hyp}(\CF)$ and consider the universal covering space \cite[Theorem 4.4]{BCGLO},
\[
\pi\to BO^p(\CF)\to B\CF,
\]
where $O^p(\CF)\leq \CF$ is the unique $p$-power index fusion subsystem of $\CF$ over $\mathfrak{hyp}(\CF)\leq S$. For $M=\BF_p[\pi]$ with $\pi$ acting by left multiplication, we get
\[
H^1(\CF^c;M)\cong H^1(B\CF;M)\cong H^1(BO^p(\CF);\BF_p)=0,
\]
as $BO^p(\CF)$ is simply connected. Hence, by hypothesis:
\[
0=H^n(\CF^c;M)\cong H^n(BO^p(\CF);\BF_p)
\]
for all $n\geq 1$. As $BO^p(\CF)$ is a $p$-complete space \cite[Proposition 1.11]{BLO2}, it must be contractible \cite[I.5.5]{BK}. Hence, $B\CF\simeq K(\pi,1)$, $\mathfrak{hyp}(\CF)=1$ and $\CF=\CF_S(S)$.
\end{proof}

\begin{remark}
This argument provides an alternative proof of the implication $(2)\Rightarrow (1)$ for finite groups. Namely, consider the following fibration of classifying spaces of finite groups,
\[
BO^p(G)\to BG\to BG/O^p(G),
\]
and then $p$-complete it \cite[II.5.1]{BK}.
\end{remark}


\begin{thebibliography}{99}

\bibitem{AKO} M.~Aschbacher, R.~Kessar, B.~Oliver, \emph{Fusion systems in algebra and topology}, London Mathematical Society Lecture Note Series, 391. Cambridge University Press, Cambridge, 2011.

\bibitem{BESW} A.~Ballester-Bolinches, L.~Ezquerro, N.~Su, Y.~Wang, \emph{On the focal subgroup of a saturated fusion system.} J.\ Algebra \textbf{468} (2016), 72--79.

\bibitem{BGH} D.J.~Benson, J.~Grodal, E.~Henke, \emph{Group cohomology and control of $p$-fusion.} Invent.\ Math.\ \textbf{197} (2014), 491--507.

\bibitem{BK} A.K.~Bousfield, D.M.~Kan, \emph{Homotopy limits, completions and localizations.} 
Lecture Notes in Mathematics, Vol. \textbf{304}. Springer-Verlag, Berlin-New York, 1972.

\bibitem{BCGLO} C.~Broto, N.~Castellana, J.~Grodal, R.~Levi, B.~Oliver, \emph{Extensions of p-local finite groups.} Trans. Amer. Math. Soc. \textbf{359} (2007), no. 8, 3791-3858. 

\bibitem{BLO2} C.~Broto, R.~Levi, B.~Oliver, \emph{The homotopy theory of fusion systems.} J. Amer. Math. Soc. 16 (2003), no. \textbf{4}, 779–856. 

\bibitem{CSV} J.~Cantarero, J.~Scherer, A.~Viruel, \emph{Nilpotent $p$-local finite groups.} Ark.\ Mat.\ \textbf{52} (2014), 203--225.

\bibitem{C} A.~Chermak, \emph{Fusion systems and localities.} Acta Math. \textbf{211} (2013), no. 1, 47-139. 

\bibitem{D} A.~D\'iaz, \emph{A spectral sequence for fusion systems.} Algebr. Geom. Topol. \textbf{14} (2014), no. 1, 349-378.

\bibitem{DGPS} A.~D\'iaz, A.~Glesser, S.~Park, R.~Stancu, \emph{Tate's and Yoshida's theorems on control of transfer for fusion systems.} J.\ Lond.\ Math.\ Soc.\ \textbf{84} (2011), 475--494.

\bibitem{GRV} J.~Gonz\'alez-S\'anchez, A.~Ruiz, A.~Viruel, \emph{On Thompson's $p$-complement theorems for saturated fusion systems.} Kyoto J.\ Math.\ \textbf{55} (2015),  617--626.

\bibitem{HRZ} K.~Hoechsmann, P.~Roquette, H.~Zassenhaus, \emph{A cohomological characterization of finite nilpotent groups.} Arch. Math. (Basel) \textbf{19} 1968 225-244.

\bibitem{KLN} R.~Kessar, M.~Linckelmann, G.~Navarro, \emph{A characterisation of nilpotent blocks.} Proc.\ Amer.\ Math.\ Soc.\ \textbf{143} (2015), 5129--5138.

\bibitem{LZ} J.~Liao, J.~Zhang, \emph{Nilpotent fusion systems.} J.\ Algebra \textbf{442} (2015), 438--454.

\bibitem{M} R.~Molinier, {\it Cohomology with twisted coefficients of the classifying space of a fusion system}, Topology Appl. 212 (2016), 1-18. 

\bibitem{W} W.J.~Wong, \emph{A cohomological characterization of finite nilpotent groups.}  
Proc. Amer. Math. Soc. \textbf{19} 1968 689-691.


\end{thebibliography}
\end{document}